\documentclass[a4paper, 12pt, twoside]{article} 

	\usepackage{verbatim}
	\usepackage[left=3cm,right=2.5cm,bottom=3.5cm,top=2.5cm]{geometry} 
	\usepackage{amsfonts}
	\usepackage[utf8]{inputenc} 
	\usepackage[english]{babel} 
	\usepackage{amssymb}        
	\usepackage{amsmath}        
	\usepackage{amsthm}         
	\usepackage{mathtools}      
	
    	\usepackage{enumitem}       
    	\usepackage{xcolor}         
    	\usepackage{fancyhdr}       
    	\usepackage{setspace}       
    \usepackage{bbm} 
    \usepackage{hyperref} 
    
	\usepackage{graphicx} 
	\usepackage{multirow}
	\usepackage{diagbox}
	\usepackage{hhline}
	\usepackage{caption}        

        \usepackage[authoryear]{natbib}

	\usepackage{subcaption}
        \usepackage{makecell}       
     
	\usepackage{algorithm}      
    
	\newcommand{\ts}{\thinspace} 
     	\newcommand{\tod}{\xrightarrow{d}}
	\newcommand{\topr}{\xrightarrow{\mathbb{P}}}

	\newcommand{\R}{\mathbb{R}}

	\newtheorem{Def}{Definition}[section]       

        \newtheorem{Thm}[Def]{Theorem}
		
        \newtheorem{rem}[Def]{Remark}

       \newcommand{\HD}{\textcolor{red}} 
       \newcommand{\lk}{\textcolor{blue}} 
        
        \DeclarePairedDelimiter\abs{\lvert}{\rvert}%
        \DeclarePairedDelimiter\norm{\lVert}{\rVert}%
        
        \makeatletter
        \let\oldabs\abs
        \def\abs{\@ifstar{\oldabs}{\oldabs*}}
        \let\oldnorm\norm
        \def\norm{\@ifstar{\oldnorm}{\oldnorm*}}
        \makeatother
        
        \hypersetup{
            colorlinks,
            citecolor=black,
            filecolor=black,
            linkcolor=black,
            urlcolor=black
        }

\title{Testing for similarity of dose response in multi-regional clinical trials}
\author{  Holger Dette$^2$ \and Lukas Koletzko $^2$  \and Frank Bretz$^1$ }
\date{\begin{small}
    $^1$\textit{Novartis Pharma AG}\\%
    $^2$\textit{Ruhr-Universität Bochum}\\[5ex]%
    \end{small}
    \date{}
}

\begin{document}
\maketitle

\begin{abstract}
    This paper addresses the problem of deciding whether the dose response relationships between subgroups and the full population in a multi-regional trial are similar to each other. Similarity is measured in terms of the maximal deviation between the dose response curves. We consider a parametric framework and develop two powerful bootstrap tests for the similarity between the  dose response curves of one subgroup and the full population, and for the similarity between the dose response  curves of several subgroups and  the full population. We prove the validity of the tests, investigate the finite sample properties by means of a simulation study and finally illustrate  the methodology in  a case study.
\end{abstract}

\section{Introduction}
  \def\theequation{1.\arabic{equation}}	
  \setcounter{equation}{0}

A multi-regional clinical trial is a single study carried out simultaneously across various regions under a common protocol to investigate the effect of an investigational drug. Its primary goal is often to draw conclusions about the drug's effect across all the regions participating in the trial. Conducted within the framework of a global drug development program, a multi-regional clinical trial aimed at bridging purposes can facilitate the drug's registration across all involved regions. In recent years, such trials have received increasing attention for their potential to reduce resources by avoiding the need for multiple, regional trials \citep{ich2017}. Accordingly, multi-regional clinical trials typically have at least two main objectives: demonstrating the drug's efficacy within individual regions and comparing study results across regions to confirm that the drug's effects are not influenced by ethnic or other regional factors.

The \cite{ich2017} guideline on general principles for planning and design of multi-regional clinical trials highlights the importance of identifying intrinsic and/or extrinsic ethnic factors that could impact drug responses in the early stages of drug development, prior to the design of confirmatory multi-regional studies. Consequently, recent early-phase studies, especially those focusing on dose response, are increasingly being conducted across multiple countries or regions \citep{song2019strategic}. The primary aim of multi-regional dose response studies is to establish the dose response relationship using data from the full (i.e., global) population. Once this primary goal is achieved, it becomes important to evaluate whether the results applicable to the full population can be reliably extended to specific regions. To this end, the dose response relationship observed in any given region should be similar to that of the full population. If significant discrepancies are observed, further investigation into the intrinsic and/or extrinsic ethnic factors affecting these outcomes may be necessary. 

\medskip

Many authors consider the problem of how to choose the subgroup sample sizes in multi-regional confirmatory trials in order to allow observing a consistent treatment effect between a regional subgroup and the full population with acceptable probability  \citep[see, among others][]{liao2018sample,teng2018practical,chiang2019use}. 
More recently, \cite{yamaguchi2021sample} and \cite{kaneko2023method} address the problem of sample size allocation 
 for demonstrating consistency (or similarity)
in multi-regional dose finding trials. In particular, \cite{yamaguchi2021sample} measure consistency between dose response profiles of a regional subgroup and the global population by  the probability that the maximal deviation between the two curves falls below a fixed threshold.  In order to facilitate the sample size calculations, \cite{kaneko2023method} restricts the calculation of the maximum deviation to the used dose levels 
and derives an approximation formula for the proposed consistency probability. 

\medskip

Equivalence tests pursue a similar approach 
and are a common tool to decide whether parameters of interest such as the area under the curve
or the peak concentration  of two groups are similar. Here, 
the   null hypothesis is defined as an effect exceeding a given threshold  and under the alternative the  effect is smaller  than this bound. Meanwhile  there 
exist well developed methodology on testing the equivalence of finite dimensional parameters
\cite[see, for example][and the references therein]{wellek2010testing}. While a large body of this literature refers to 
applications in medicine, in particular pharmacokinetics,  equivalence tests have also been used in other areas such as economics, psychology or biology \citep[see, for example,][]{KimRobinsonAndrew2019,lakens2020equivalence,ROSE201877}. A common feature of all these references consists in the fact that equivalence refers  to finite  dimensional (often one-) dimensional parameters, which should be close with respect to some metric.

The problem of investigating the similarity between curves, however, is far less explored.
\cite{cade2011} and \cite{ostrovski2022} develop tests for the similarity of quantile curves and power law distributions, respectively.
In the context of drug development,  \cite{liuhaywyn2007, liubrehaywynn2009, gsteiger2011} and \cite{bretz2018assessing} develop tests 
for the similarity between dose response curves of two distinct groups of patients utilizing a parametric model assumption and  confidence  bands for the difference  between the two curves. More recently, \cite{dette2018equivalence, mollenhoff2020equivalence, moedetbre2021} propose 
more powerful similarity tests based on   constrained parametric bootstrap. For a Bayesian approach to the problem we also mention \cite{ollier2021estimating} and the references therein. A common feature in all these references consists in the fact that  parameters or curves corresponding to two different and independent groups are compared.

\medskip

In this paper, we consider the problem of assessing whether the dose response relationships of one or more (regional) subgroups and the full population are similar from a hypothesis testing point of view. Similar to much of the aforementioned literature, we define a parametric model for the clinical trial data and formulate appropriate statistical hypotheses that capture the similarity problem. We propose  powerful parametric bootstrap tests for the similarity between the dose response curves of one subgroup and the full population   and for the similarity between the dose response curves of several subgroups and the full population.
The validity of these procedures is proved rigorously and the performance is analyzed in various scenarios by means of a simulation study which includes small sample sizes. Finally, we illustrate the use of these tests in a multi-regional dose finding case study. 
Our work differs from  that of \cite{yamaguchi2021sample} and \cite{kaneko2023method} with respect to several aspects. First, the  dose response functions of the subgroups in our model are not restricted to be identical. 
In fact, they are allowed to have different parametric forms. Second, our model allows not only for two, but also for more than two subgroups and  we also propose a test for simultaneously assessing the similarity between several subgroups and the full population. Third, we compare the dose response curve on the full dose range and not only at the dose levels used in the clinical trial. Fourth, we develop a constrained bootstrap test to obtain critical values and $p$-values.

\medskip

The remainder of this paper is organized as follows. In Section \ref{sec2} we introduce the framework and methodology for assessing similarity of one or several subgroups with the full population. Section \ref{sec3} and Section \ref{example} are dedicated to a simulation study and numerical example, respectively. We conclude with a discussion in Section \ref{disc} and give all mathematical details in the Appendix in Section \ref{appendix}. 

\section{Assessing similarity of 
subgroups with the  full population}
\label{sec2}
  \def\theequation{2.\arabic{equation}}	
  \setcounter{equation}{0}

Suppose a dose response trial is conducted with $r$ dose levels $d_1, \ldots , d_r$ in the dose range  $\mathcal{D}=[d_1, d_r]$, where $d_1 = 0$ denotes the placebo group. Assume that a population  of patients can be decomposed into $k$ disjoint subgroups (corresponding to the  different regions in a multi-regional clinical trial). Let $n_\ell$ 
denote the number of patients belonging  to subgroup  $\ell = 1, \ldots , k$,  and denote by   $n = n_1 + \ldots + n_k$ the total number of patients recruited for the trial. Each patient is randomized to one of the dose levels $d_1, \ldots , d_r $. Let   $n_{\ell, j}$ denote the number of patients in subgroup $\ell$ which are treated at dose level $d_j$ ($j=1, \ldots , r$). Following \cite{bretz2005combining}, \cite{thomas2006hypothesis}, 
\cite{bretz2018assessing, mollenhoff2020equivalence, yamaguchi2021sample}, we assume that the  dose response relationships in each  subgroup can be described by a (possibly non-linear) parametric dose response model, say  $\mu_{\ell}(\cdot, \beta_\ell): \mathcal{D} \to \R $ with a $\gamma_\ell$-dimensional parameter $\beta_\ell$. 
We model the response of the $i$th patient treated with dose $d_j$
in  subgroup $\ell$ as a normal distributed random variable 
 with  variance $\sigma_\ell^2 > 0$ and mean $\mu_{\ell}(d_j, \beta_\ell)$, that is
 \begin{align}
    \label{data1}
    Y_{\ell i j} = \mu_{\ell}(d_j, \beta_\ell) + \epsilon_{\ell ij}, \quad  i=1, \ldots , n_{\ell, j} , ~j=1, \ldots , r, ~\ell =1, \ldots , k ~, 
\end{align}
where $\epsilon_{\ell ij}$  are independent centered normal distributed errors with variance $\sigma^2_\ell >0$. This means that  we  assume the patient responses to be independent. Note that the dose response models $\mu_\ell$ and variances $\sigma_\ell^2$ 
are  allowed to be different for the $k$ different subgroups. We are interested in testing whether the mean effect in a particular subgroup  is similar to the mean effect in the full population. For this purpose  we model the mean treatment effect in the full population as a weighted average of the regional treatment effects, where the weights represent the share of each region in the global or overall  effect. Similar modelling approaches are employed by \cite{bean2023bayesian}  and \cite{kaneko2023method} among others.
Following these authors we assume that the $k$  subgroup proportions in the full population are known and denote these
by  $p_1, \ldots , p_k$, where $p_\ell$ represents the positive proportion of the $\ell$th subgroup ($\sum_{\ell =1}^k p_\ell =1$). We  define an overall (population) effect at dose $d$  by 
\begin{align}
\label{det1}
\bar\mu (d, \beta) := \sum^k_{\ell = 1} p_\ell  \mu_\ell (d, \beta_\ell)~, 
\end{align}where  $\beta = (\beta_1^\top, \ldots , \beta_k^\top)^\top$
denotes the vector of all parameters in the regression models $\mu_1. \ldots , \mu_k$
corresponding to the different subgroups.
With these notations we can investigate the problem of testing similarity between the dose response curves of one or more subgroups and the full population.

\subsection{Assessing similarity of one subgroup with the  full population} \label{sec21}

Without loss of generality we assume that the dose response curve of the first  subgroup $\ell=1$  has to be compared with the full population. We will address this problem by estimating the maximum deviation 
\begin{align} \label{det2}
    d_{\infty} := d_{\infty}(\beta) := \max_{ {d} \in  \cal{D} } | \mu_1 (d, \beta_1) - \bar\mu {(d, \beta)} |
\end{align}
between the (expected) dose response curve $\mu_1$ in subgroup $1$ and the dose response curve $\bar \mu$ defined in \eqref{det1}.
In the literature, maximum deviation distances of the form \eqref{det2} are considered to investigate the similarity of curves from two independent groups, see for example \cite{liubrehaywynn2009, gsteiger2011, dette2018equivalence, moedetbre2021}. In the context of multi-regional clinical trials, \cite{yamaguchi2021sample} use this distance to study differences between  the dose response curves of a subgroup and the full  population consisting of $k=2$ subgroups.

In order to investigate if the observed dose response relationship of a specific subgroup (here the first one) is  sufficiently similar to that of the full population we will develop  a test for the hypotheses
\begin{align}
\label{hyp1}
    H_0: d_{\infty}  \geq \Delta ~ \text{versus} ~ H_1:  d_{\infty}  < \Delta,
\end{align}
where  $\Delta > 0$  is a given threshold that
depends on the clinical relevance in a particular application.
Note that rejecting $H_0$ in \eqref{hyp1} 
means to decide 
that the absolute difference between the dose response curves of the regional subgroup and the full population is smaller than $\Delta$ over the whole dose range while keeping the probability for a type 1 error bounded by the significance level, say $\alpha$.

\begin{rem}
\label{rem1}
{ \normalfont
In the special case $k=2$, the hypotheses in \eqref{hyp1} reduce to the hypotheses  considered in \cite{yamaguchi2021sample} (up to a constant factor). Note that in this case $p_2=1-p_1$ and the  curve for the full  population \eqref{det1} reduces to 
\begin{align*}
    \bar\mu(d, \beta) =  p_1 \ts \mu_{1}(d, \beta_1) + (1-p_1) \ts \mu_{2}(d, \beta_2), 
\end{align*}
which yields for the maximum deviation distance in \eqref{det2} the representation
\begin{align*}
    d_{\infty} = (1-p_1) \max_{ {d} \in  \cal{D} }|  \mu_1 (d, \beta_1) - \mu_{2}(d, \beta_2)  |.
\end{align*}
In this case the alternative hypothesis in \eqref{hyp1} coincides with statement (10) in \cite{kaneko2023method}. 
}
\end{rem}

To estimate the unknown maximal deviation  $d_\infty$ in \eqref{det2} we use a plug-in estimator defined by
\begin{align}
    \label{estim1}
   \hat d_{\infty}  := d_{\infty}(\hat \beta) = \max_{ {d} \in  \cal{D} } | \mu_1 (d, \hat \beta_1) - \bar\mu {(d, \hat \beta)} |,
\end{align}
where $\hat \beta =  ( \hat \beta_1^\top ,  \ldots , \hat \beta_k^\top)^\top$ and  $ \hat \sigma^2 = ( \hat \sigma_1^2 ,  \ldots  , \hat \sigma_k^2)^\top $
are the maximum likelihood estimators (mle) of the parameters
 $\quad \beta = (\beta_1^\top, \ldots , \beta_k^\top)^\top$ and 
 $ \sigma^2 = (\sigma_1^2, \ldots  , \sigma_k^2)^\top$, respectively, 
 maximizing
the log-likelihood function 
\begin{align}
\label{mle}
   \log( L(\beta, \sigma^2) ) =  -\sum_{\ell = 1}^{k} \sum_{j=1}^{r} \sum_{i = 1}^{n_{\ell, j}}\Big \{ \log( (2 \pi \sigma_\ell^2)^{1/2} ) +  \dfrac{1}{2 \sigma_\ell^2} (Y_{\ell i j } - \mu_\ell(d_j, \beta_\ell ) )^2  \Big\} .
\end{align}
We reject the null hypothesis in \eqref{hyp1} for small values of the statistic $\hat d_{\infty} $. However, the distribution of $\hat d_{\infty} $ under the null is complicated and critical values are difficult to obtain (see Theorem \ref{thm2} in the Appendix).  
To address this problem, we propose to use a non-standard 
constrained parametric bootstrap test for the hypotheses \eqref{hyp1}. The pseudo-code for this procedure is given in Algorithm \ref{alg2}.

\begin{algorithm}[H]
\small 

\caption{Constrained parametric bootstrap test for hypotheses \eqref{hyp1}}
                \label{alg2}

       		\begin{itemize}
       			\item[(1)] Calculate the  mle $ (\hat{\beta}, \ \hat \sigma^2) $ 
          and test statistic $\hat{d}_{\infty}$ in \eqref{estim1}.
       			\item[(2)] Calculate a constrained version of the mle of $\beta$ defined by
       			\begin{align*}
       			 \hat{\hat{\beta}}  = \begin{cases}
       			\hat{\beta}, \ \hat{d}_{\infty} \geq \Delta \\
       		\tilde{\beta}, \ \hat{d}_{\infty} < \Delta.
       			\end{cases}
       			\end{align*}
                    where $\tilde{\beta}$ is the mle of $\beta$ in the set $\{ \beta : d_{\infty}(\beta) = \Delta \}.$
       			\item[(3)] For $\ell = 1, \ldots , k$, $j=1, \ldots , r$, $i=1, \ldots , n_{\ell, j}$ 
          generate  bootstrap data 
                \begin{align*}
                    Y_{\ell ij}^{*} = \mu_{\ell}(d_j, \hat{\hat{\beta}}) + \epsilon_{\ell ij}^*,
                \end{align*}
                where $\epsilon_{\ell ij}^*$ are independent centered normal distributed with variance $ \hat \sigma^2_\ell $.
       			
  \item[(4)] Calculate  the mle $\hat \beta^*$  from the bootstrap data 
  $\{ Y_{\ell ij}^{*} \}_{\ell = 1, \ldots , k; j=1, \ldots , r;  i=1, \ldots , n_{\ell, j} } $
  and 
    the $\alpha$-quantile  $\hat{q}_{\alpha}^*$  of the distribution of
   \begin{align}
       \label{det23}
   \hat{d}_{\infty}^{*} := d_{\infty}(\hat \beta^* ) = \max_{ {d} \in  \cal{D} } | \mu_1 (d, \hat \beta_1^*) - \bar\mu {(d, \hat \beta ^*)} |.
   \end{align}
  
  \item[(5)] Reject the null hypothesis in \eqref{hyp1}, whenever
       	\begin{align}
       	\label{boot}
            \hat{d}_{\infty} < \hat{q}_{\alpha}^{*}. 
       	\end{align}
       		
\end{itemize}

\end{algorithm}

\begin{rem} ~~
{\rm  
\begin{itemize}
\item[(a)] In practice, the quantile $\hat q_\alpha^*$ in Algorithm \ref{alg2} is simulated by the empirical $\alpha$-quantile of $B$ realizations of the bootstrap statistic $\hat d_\infty^*$ in \eqref{det23}. More precisely, if $ \hat{d}_{\infty}^{*,(1)}, \ldots , \hat{d}_{\infty}^{*,(B)} $ denote 
$B$ independent bootstrap copies of $\hat{d}_{\infty}^{*}$ generated by Algorithm \ref{alg2}, the empirical $\alpha$-quantile of this sample, say $\hat{q}_{\alpha}^{*,B}$, is used as an estimate of $\hat{q}_{\alpha}^{*}$. 
In a similar way we can define a $p$-value for testing the hypotheses \eqref{hyp1} by 
\begin{align}
\label{pvalue1}
    p_{ \infty } := \tfrac{1}{B} \sum_{b=1}^{B} \mathbbm 1 \{ \hat d_\infty^{*,(b)} \leq \hat d_\infty \}.
\end{align}

    \item[(b)]
In   Theorem \ref{boot1} in the Appendix we establish the validity of this bootstrap test.  For  sufficiently large
 sample sizes the test \eqref{boot} keeps its nominal level $\alpha$. 
To be more precise,  consider the set\begin{align*}
                {\cal E} = \{ d \in \mathcal{D} : |\mu_1(d, \beta_1) - \bar \mu(d, \beta) | = d_\infty  \},
   \end{align*}
    which consists of all dose levels 
    where the (absolute) difference between the two dose response curves is maximal.
   In most applications the 
    set ${\cal E}$ has only one element (see Section \ref{sec3} for some typical non-linear regression models used in dose response trials) and in such situations the properties of the test can be easily described. If $d_\infty >\Delta  $  (we call this region  {\it interior of the null hypothesis in \eqref{hyp1}})  the probability of rejection  converges to $0$ for increasing sample sizes.
      If $d_\infty = \Delta  $  (we call this  region  {\it boundary  of the   hypotheses  in \eqref{hyp1}})  the probability of rejection  converges to the nominal level  $\alpha$ for increasing sample sizes.  Moreover, the test \eqref{boot} detects the  alternative with a probability converging to $1$ with increasing sample size, which means that it is consistent.
      A rigorous formulation of this result and a discussion of the case where $\mathcal{E}$ consists of more than one point can be found in Theorem \ref{boot1}. The finite sample properties of the test are investigated in Section \ref{sec31} by means of a simulation study.
     \item[(c)]
     As pointed out in Remark \ref{rem1}, 
  \cite{yamaguchi2021sample} and \cite{kaneko2023method}  assess the similarity  between the dose response curves of one subgroup and the full population consisting of only two subgroups. However, even in this case the testing approach  considered in the present paper  differs from their method as follows.  \cite{yamaguchi2021sample} and \cite{kaneko2023method} consider the size of a ``consistency" (or ``assurance") probability $ \mathbb P( \hat d_\infty < \Delta) $ 
 which is calculated under the assumption that the two curves are 
 exactly identical. In particular, they do not aim for a control of the probability 
 $\mathbb P( \hat d_\infty \geq  \Delta) $ of a type I error, in contrast to the constrained bootstrap test \eqref{boot}. 
 Consequently, in the case of two groups, their test cannot be compared directly with the one proposed in this paper as it is not calibrated at the correct nominal level. For example, Table 6  in \cite{kaneko2023method} has to be interpreted with some care as the different tests under  consideration are not calibrated for the same type I error rate.
\end{itemize}
}
\end{rem}

\begin{rem}
\label{rem2}
{\rm  
    Note that the hypotheses  in \eqref{hyp1} are nested.
    Recalling the definition of the 
    bootstrap  in Algorithm \ref{alg2} 
    it is easy to see that $\hat{d}_{\infty , \Delta_1}^{*} \leq \hat{d}_{\infty , \Delta_2}^{*}$, where $\Delta_1 \leq \Delta_2$ and $\hat{d}_{\infty ,\Delta}^{*}$ denotes the bootstrap statistic \eqref{det23} calculated by  Algorithm \ref{alg2} for the threshold $\Delta$. Consequently, we obtain for the corresponding quantiles the inequality $\hat{q}_{\alpha , \Delta_1}^{*} \leq \hat{q}_{\alpha , \Delta_2}^{*}$, and 
  rejecting the null hypothesis in \eqref{hyp1}  
  by the test \eqref{boot}
  for $\Delta= \Delta_0$ also yields  rejection of the null 
  for all  $\Delta>\Delta_0$.
\\
 Therefore,  by the sequential rejection principle, we may simultaneously test the  hypotheses  in \eqref{hyp1} for different $\Delta \geq  0$   starting at $\Delta  = 0$ and 
   increasing  $\Delta $ to 
   find the minimum value $ \hat \Delta_\alpha  $
   for which 
   $H_0$ is rejected for the first time. This value could be interpreted as a measure of evidence  for similarity with a controlled type I error  $\alpha$.
}
\end{rem}

\subsection{Assessing similarity of several subgroups with the  full population}

In this section we extend the methodology for investigating the similarity of $1 \leq m \leq k $ subgroups 
with the full population. Without loss of generality we assume  that we are interested in dose response curves  
corresponding to the  subgroups $1, \ldots , m$ and consider the distance
\begin{align}
\label{maxall}
    d_{\infty, \infty} := d_{\infty, \infty}(\beta) := \max_{ 1 \leq i  \leq m } \max_{d \in \mathcal{D}}  | \mu_i (d, \beta_i) - \bar\mu {(d, \beta)} |.
\end{align}
 In order to establish simultaneously 
 the similarity of the dose response curves in the subgroups $1, \ldots ,m $ with the dose response curve of the full population
 we consider the  null hypothesis
\begin{align}
\label{hyp4}
    H_0: d_{\infty, \infty}  \geq \Delta \quad \text{versus} \quad H_1:d_{\infty, \infty} < \Delta
\end{align}
for a pre-specified threshold $\Delta > 0$. By the intersection-union principle \citep[see, for example, ][]{Sonnemann}, it can be tested by applying the test \eqref{boot} from Section \ref{sec21} for each subgroup $i = 1, ... , m$. The null hypothesis \eqref{hyp4} is then rejected if and only if all individual tests reject  the individual null hypotheses of similarity between the $i$th curve and the dose response curve of the full population. 

However, as tests based on the intersection-union principle can be conservative, we propose an alternative, more powerful test in the following.
To this end, recall the definition of the estimators $\hat \beta_1 ,\ldots  , \hat \beta_m$ and $\hat \beta $ in Section \ref{sec21}. We estimate $d_{\infty, \infty}$ by  
\begin{align}
\label{det22}
\hat  d_{\infty, \infty} := d_{\infty, \infty}(\hat \beta) = \max_{ 1 \leq i  \leq m } \max_{d \in \mathcal{D}}  | \mu_i (d, \hat\beta_i) - \bar\mu {(d, \hat \beta)} |
\end{align}
and reject the null hypothesis in \eqref{hyp4} for small values of   $\hat  d_{\infty, \infty}$.
The corresponding quantiles are obtained by the  constrained parametric bootstrap test in Algorithm \ref{alg3}, where the decision rule is defined by \eqref{allboot}. In practice, the quantile $\hat{q}_{\alpha , \infty }^*$ in \eqref{allboot} is estimated by the empirical $\alpha$-quantile of the bootstrap sample
$\hat d_{\infty, \infty}^{*,(1)}  , \ldots , \hat d_{\infty, \infty}^{*,(B)}$, where 
for $b=1, \ldots , B$  the quantity 
$\hat d_{\infty, \infty}^{*,(b)}$ is generated by Algorithm \ref{alg3}. The $p$-value for testing the hypotheses \eqref{hyp4} is defined by 
\begin{align}
\label{pvalue2}
 p_{ \infty, \infty } := \tfrac{1}{B} \sum_{b=1}^{B} \mathbbm 1 \{ \hat d_{\infty, \infty}^{*,(b)} \leq \hat d_{\infty, \infty} \}.
\end{align}
In Section \ref{alg2proof} of the Appendix we show that the decision rule \eqref{allboot} defines a valid test for the hypotheses \eqref{hyp4}. The finite sample properties of this test are investigated in Section \ref{sec32} by means of a simulation study.

\begin{algorithm}[H]
\small 

\caption{Constrained parametric bootstrap test for hypotheses \eqref{hyp4}}
                \label{alg3}
       		\begin{itemize}
       			\item[(1)] Calculate the mle $ (\hat{\beta}, \ \hat \sigma^2) $ and the test statistic $\hat{d}_{\infty, \infty}$ in \eqref{det22}.
       			\item[(2)] Calculate a constrained version of the mle of $\beta$ defined by
       			\begin{align*}
       			 \hat{\hat{\beta}}  = \begin{cases}
       			\hat{\beta}, \ \hat{d}_{\infty, \infty} \geq \Delta \\
       		\tilde{\beta}, \ \hat{d}_{\infty, \infty} < \Delta.
       			\end{cases}
       			\end{align*}
                    where $\tilde{\beta}$ is the mle of $\beta$ in the set $\{ \beta : d_{\infty, \infty}(\beta) = \Delta \}.$
       			\item[(3)] For $\ell = 1, \ldots , k$, $j=1, \ldots , r$, $i=1, \ldots , n_{\ell, j}$ 
          generate  bootstrap data 
                \begin{align*}
                    Y_{\ell ij}^{*} = \mu_{\ell}(d_j, \hat{\hat{\beta}}) + \epsilon_{\ell ij}^*,
                \end{align*}
                where $\epsilon_{\ell ij}^*$ are independent, centered normal distributed with variance $ \hat \sigma^2_\ell $.
       			
  \item[(4)] Calculate  the mle $\hat \beta^*$  from the bootstrap data 
  $\{ Y_{\ell ij}^{*} \}_{\ell = 1, \ldots , k; j=1, \ldots , r;  i=1, \ldots , n_{\ell, j} } $
  and the $\alpha$-quantile  $\hat{q}_{\alpha , \infty }^*$  of the  distribution of
    $$ \hat  d_{\infty, \infty}^{*} := d_{\infty, \infty}( \hat \beta^* ) = \max_{ 1 \leq i  \leq m } \max_{d \in \mathcal{D}}  | \mu_i (d, \hat\beta_i^*) - \bar\mu {(d, \hat \beta^*)} |.$$
  
  \item[(5)] Reject the null hypothesis in \eqref{hyp4}, whenever
       	\begin{align}
       	\label{allboot}
            \hat{d}_{\infty,\infty} < \hat{q}_{\alpha , \infty}^{*}. 
       	\end{align}
       		
\end{itemize}

\end{algorithm}

\section{Finite sample properties}

\label{sec3} 

  \def\theequation{3.\arabic{equation}}	
  \setcounter{equation}{0}
  
In this section we investigate the finite sample performance of the bootstrap procedures defined in Algorithm \ref{alg2} and Algorithm \ref{alg3} by means of a simulation study. The constrained maximum likelihood estimators $\tilde \beta$ (see step (2) in both Algorithms) are computed with the \texttt{auglag} function provided by the \texttt{alabama} package in \texttt{R}.

\subsection{Assessing similarity of one subgroup with the full population}
\label{sec31}

Our simulation setup is inspired by a multi-regional clinical trial design for a Phase II dose finding study for an anti-anxiety drug first considered in \cite{pinheiro2006design} and subsequently investigated in \cite{yamaguchi2021sample} and \cite{kaneko2023method}. 
We consider a population with $k=3$ regional subgroups, where $\ell=1,2,3$ denote the Japanese, North American, and European regions, respectively, and assume the proportions $p_1 = 0.1, \ts p_2 = 0.3$ and $ p_3 = 0.6$. We test for similarity between the full population and the Japanese subgroup ($\ell=1$) using the test \eqref{boot} in Algorithm \ref{alg2} for the hypotheses \eqref{hyp1}.
%
The dose range is $\mathcal{D} = [0, 150]$ (in mg) and we consider six dose levels $0$, $10$, $25$, $50$, $100$ and $ 150$. We fix a total sample size of $n= 450$ patients, but consider two cases for the 
sample size allocations across subgroups:
\begin{align}
\label{det3a}  &  n_1 = 150,~ n_2 = 150, ~n_3 = 150  \\
\label{det3b}   &  n_1 = 66, ~~n_2 =  192, ~ n_3 = 192
\end{align}
The motivation for these two scenarios is to investigate a balanced patient allocation in scenario \eqref{det3a} and a more realistic allocation in scenario \eqref{det3b} where the proportion of Japanese patients in the trial is smaller than those of the two other regions.
In both  scenarios \eqref{det3a} and \eqref{det3b} 
  we investigate two choices for the 
  number of patients allocated at each dose 
  level
  \begin{itemize}
      \item[$\mathcal{D}_{=}$:] In each subgroup the same number of 
patients is  treated  at each dose level 
      $0,$  $ 10,$ $ 25,$ $ 50,$ $ 100,$ and $ 150$.
            \item[$\mathcal{D}_{\neq}^{(1)}$:] In scenario \eqref{det3a} in each subgroup 
            $35, 20, 20, 20, 20$ and $ 35$
patients are treated  at dose levels 
      $0, 10, 25, 50, 100$ and $ 150$, respectively.
     \item[$\mathcal{D}_{\neq}^{(2)}$:] In scenario \eqref{det3b} in the first subgroup 15, 9, 9, 9, 9 and 15 are treated at dose levels $0, 10, 25, 50, 100$ and $ 150$, respectively, whereas in the other two subgroups 46, 25, 25, 25, 25 and 46 patients are treated at those dose levels.
  \end{itemize}
For the dose response curves in the three subgroups we consider E-max curves defined by
\begin{align}
\label{emax}
   E_0 +  \dfrac{ E_{\max} \cdot d ^h }{d^h + ED_{50}^h },
\end{align}
 where $E_0$ represents the 
 placebo effect of the drug (obtained for $d=0$), $E_{max}$ denotes the maximum effect, $ED_{50}$ is the dose which produces half of $E_{max}$ and $h$ denotes the slope (or Hill-) parameter which controls the steepness of the dose response curve. The errors in model \eqref{data1} are assumed to be centered normal distributed with standard deviation $\sigma_\ell=0.1$ $(\ell=1,2,3$).
For the mean functions we distinguish three scenarios (A), (B) and (C) (see Table \ref{scenarios}), 
\begin{table}[H]
    \centering
    \begin{tabular}{c|c|c|c|}
      \cline{2-4}
      & $\mu_1$ & $\mu_2$ & $\mu_3$ \\
      \hline
     \multicolumn{1}{|c|}{(A)} & $(0, E_{max}, ED_{50}, 1)$ & $(0, 0.46, 26, 1)$ & $(0, 0.46, 25.5, 1)$ \\
     \multicolumn{1}{|c|}{(B)} & $(0, E_{max}, 25, h)$ & $(0, 0.46, 26, 1)$ & $(0, 0.46, 25.5, 1)$ \\
     \multicolumn{1}{|c|}{(C)} & $(0, E_{max}, ED_{50}, h)$ & $(0, 0.46, 27, 2.5)$ & $(0, 0.46, 26.5, 2.5)$ \\
     \hline
    \end{tabular}
    \caption{\it Parameters $(E_0, E_{max}, ED_{50}, h)$ in the E-Max models $\mu_1$, $\mu_2$ and $\mu_3$ (see equation \eqref{emax}). The non-specified parameters for $\mu_1$ are defined in the simulation.}
    \label{scenarios}
\end{table}

where the parameters $E_{max}, ED_{50}$ and $h$ of the first curve are varied for simulating the rejection probability of the Algorithms in the interior, on the boundary of the null hypothesis and under the alternative. The choice of these parameters is inspired by the (single) E-Max candidate dose response model considered in \cite{yamaguchi2021sample} who set the parameters as  $E_0 = 0, E_{max} = 0.46, \ts ED_{50} = 25$ and $h=1$. In scenario (A) and (B) the curve for the first subgroup is chosen to deviate more and more from this reference dose response model as we vary the parameters $ED_{50}$, respectively $h$. Scenario (C) is added to also investigate the case, where both the $ED_{50}$ and Hill-coefficient of the Japanese curve vary simultaneously. Note that in all three scenarios (A), (B) and (C) the curves for the North American and European subgroup are chosen to be very close to each other as this would generally be expected in practice. To reflect reality, the $E_{max}$ parameter of the Japanese curve is also chosen very similar to the one for the North American and European subgroup (approximately 0.46) throughout all parameter choices. The curves corresponding to the different subgroups are displayed in Figure \ref{fig1}, where we show several curves for the first subgroup. 
 \medskip

The rejection probabilities of the test \eqref{boot} are displayed in Table \ref{tab1}, \ref{tab2} and \ref{tab3} 
for the scenarios (A), (B) and (C), respectively.
 Note that the $E_0$-coefficients are estimated in all scenarios.
 We begin with a discussion of the case (A) in Table \ref{tab1}.  Here we consider two  cases:  first, we assume the Hill-parameters of the curves to be  known, resulting in $9$ parameters to be estimated overall (three for each curve); second we also estimate these  parameters, yielding $12$ unknown parameters. 
The numbers in brackets represent the results for the model in (A) where the Hill-coefficient is assumed to be known ($h=1$) and is not estimated.  We observe that the test keeps its nominal level $\alpha = 10\%$ whenever $d_\infty \geq  \Delta =0.1$ and has reasonable power  for $d_\infty < \Delta =0.1$. When all four parameters of the models are estimated the test \eqref{boot} is conservative: even at the boundary $d_\infty = \Delta$  the level is smaller than  $\alpha$. On the other hand, if the Hill-coefficients in all three models 
 are assumed to be known, the simulated level at the boundary is close to  $\alpha= 10\%$.
Fixing the Hill-parameter yields also a significant improvement in power. For example, under the alternative determined by the Japanese curve with $ED_{50} = 10$ and $E_{max} = 0.42$ (third row) the power of the test more than doubles in all four dosing scenarios, if we assume the Hill-coefficients of the curves to be known. For instance, for the equidistant design $\mathcal{D}_{=}$ the power improves from 0.317 to 0.672. A comparison  of the  sample size allocations  to the different dose levels shows that an equal allocation $\mathcal{D}_{=}$
yields a more powerful test than the designs $\mathcal{D}_{\neq}$. Similarly, using the equal sample sizes \eqref{det3a} for the populations yields a more powerful test than using a non-uniform design as \eqref{det3b}.

Next, we discuss the results for scenario (B) and (C), where all four parameters in the three models are estimated. The rejection probabilities for scenario (B) are displayed in Table \ref{tab2}.
Again, the test \eqref{boot} keeps its nominal level 
$\alpha = 10 \%$. Note that in this scenario  there are two cases corresponding to the boundary $d_\infty = \Delta =0.1$ and the quality of the approximation is different in these cases. In the  case  $h= 3.5$ and $E_{\max} =0.40$ the approximation of the nominal level at the
boundary of the hypotheses is much more accurate as in the case $h=0.3$ and $E_{\max} =0.47$, 
in particular for equal sample sizes of the populations (as specified in \eqref{det3a}) and 
a corresponding equidistant design. Similarly to scenario (A) power is improved by equal 
sizes for the populations and equal sample sizes at the different dose levels.
For some cases the results can be compared with the results in Table \ref{tab1}. For example, for  $h= 1.5$ and $E_{\max} =0.43$ we obtain $d_\infty = 0.04$, which corresponds to the case $ED_{50}=15$ and $E_{\max} = 0.44$ in Table \ref{tab1}.  In this case the alternative in scenario (B) is easier to detect than the alternative in (A), although both cases yield a maximal deviation $d_\infty = 0.04$. These observations indicate that  the power of the test \eqref{boot} is not completely determined by the distance 
$d_\infty $ but also depends on the properties of the curves.
The results in Table \ref{tab3} for the scenario (C) show a similar picture  as for cases (A) and (B) and confirm our findings. 

We mention again that, in  all three scenarios, the test (generally) performs best if the sample sizes for the subgroups are identical and the patients are allocated uniformly to the different dose levels. These  results suggest that, in order to improve the  power of the equivalence test \eqref{boot}, the best strategy is  to choose the subgroup (and dose group) sizes as uniform as possible. However, we emphasize that this rule of thumb is only applicable, if the variances in all groups and all dose levels are similar.

\begin{figure}[H]
    \centering
    \includegraphics[width=7.5cm, height=7.5cm]{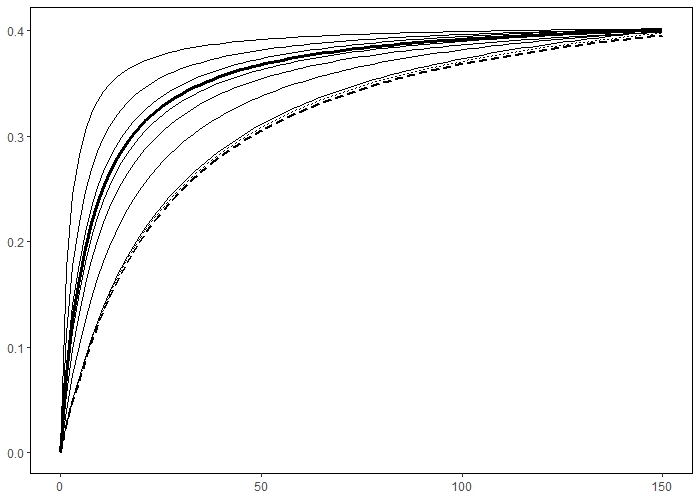}
    ~~
    \includegraphics[width=7.5cm, height=7.5cm]{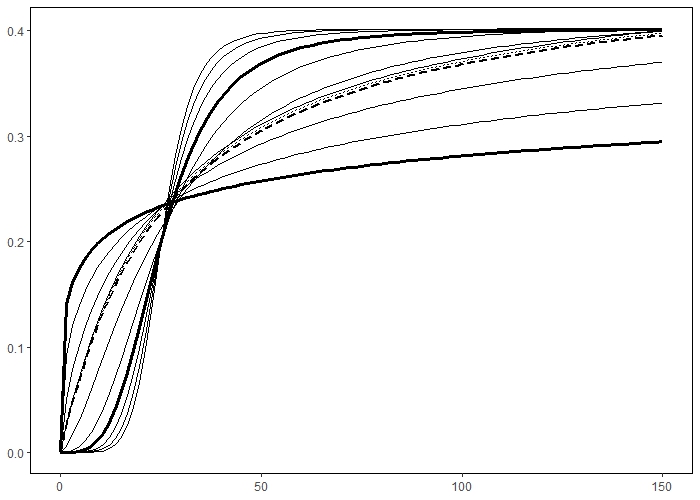}
    ~~
    \includegraphics[width=7.5cm, height=7.5cm]{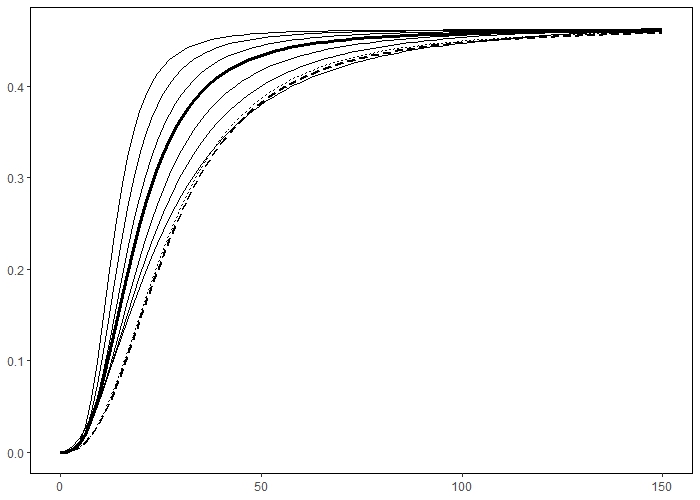}
    \caption{\it The dose response curves considered in the simulation study. Top row: scenario (A) (left) and scenario (B) (right); bottom row: scenario (C). The (different) solid  curves correspond to  the first subgroup (for various parameters $E_{\max}$, $ED_{50}$ and $h$), whereas the dashed and dotted curve correspond to the second and third subgroup, respectively. The thick solid curves satisfy $d_{\infty} = \Delta = 0.1$ and mark the boundary of the null hypothesis set.  
    }
    \label{fig1}
\end{figure}


\begin{table}[H]
    \centering
    \begin{tabular}{cc|c|c || c|c|}
    \cline{3-6}
     &  & \multicolumn{2}{c||}{ \eqref{det3a} } & \multicolumn{2}{c|}{ \eqref{det3b} } \\
     \hline
    \multicolumn{1}{|c}{$(ED_{50}, E_{max})$} & $d_{\infty}$ &  $\mathcal{D}_{=}$ & $\mathcal{D}_{\neq}^{(1)}$ & $\mathcal{D}_{=}$ & $ \mathcal{D}_{\neq}^{(2)}$ \\
       \hline
            \multicolumn{1}{|c}{(25, 0.47)}         & 0.00          &  0.996 (1.000)          &  0.994 (1.000)  &  0.954 (0.986) &  0.952 (0.996)            \\
            \multicolumn{1}{|c}{(15, 0.44)}         & 0.04          &  0.884 (0.990)          &  0.838 (0.972)  &  0.704 (0.898) &  0.670 (0.892)           \\
            \multicolumn{1}{|c}{(10, 0.42)}         & 0.07          &  0.317 (0.672)          &  0.278 (0.590)  &  0.224 (0.450) &  0.198 (0.448)            \\
            \multicolumn{1}{|c}{(8, 0.42)}          & 0.09          &  0.077 (0.228)          &  0.078 (0.206)  &  0.090 (0.200) &  0.062 (0.192)            \\ 
            \multicolumn{1}{|c}{\textbf{(7, 0.42)}} & \textbf{0.10} &  \textbf{0.027 (0.094)} &  \textbf{0.034 (0.082)}  &  \textbf{0.030 (0.092)}  &  \textbf{0.022 (0.098)}             \\ 
            \multicolumn{1}{|c}{(6, 0.42)}          & 0.11          &  0.007 (0.018)          &  0.014 (0.022)  &  0.014 (0.020) &  0.010 (0.046)            \\
            \multicolumn{1}{|c}{(4, 0.41)}          & 0.14          &  0.000 (0.000)          &  0.000 (0.002)  &  0.000 (0.000) &  0.002 (0.000)            \\
            \multicolumn{1}{|c}{(2, 0.40)}          & 0.19          &  0.000 (0.000)          &  0.000 (0.000)  &  0.000 (0.000) &  0.000 (0.000)            \\
            \hline
    \end{tabular}
    \caption{ \it Scenario (A): Simulated rejection probabilities of the test \eqref{boot} in Algorithm \ref{alg2} for subgroup standard deviation $\sigma_\ell = 0.1$, significance level $\alpha = 0.1$ and equivalence threshold $\Delta = 0.1$ for different sample sizes of the subgroups and different dose group sizes. Results in brackets are obtained without estimation of the Hill-parameters of the curves.}
    \label{tab1}
\end{table}

\begin{table}[H]
    \centering
    \begin{tabular}{cc| c|c||c|c|}
    \cline{3-6}
    &  & \multicolumn{2}{c||}{ \eqref{det3a} } & \multicolumn{2}{c|}{ \eqref{det3b} } \\
     \hline
       \multicolumn{1}{|c}{$(h, E_{max})$} & $d_{\infty}$ &  $\mathcal{D}_{=}$ & $\mathcal{D}_{\neq}^{(1)}$ & $\mathcal{D}_{=}$ & $ \mathcal{D}_{\neq}^{(2)}$ \\
       \hline 
            \multicolumn{1}{|c}{\textbf{(0.3,  0.47)}} & \textbf{0.10}  &  \textbf{0.068}  &  \textbf{0.072}   &  \textbf{0.072}   &  \textbf{0.09}                    \\
            \multicolumn{1}{|c}{(0.5,  0.47)}          & 0.06           &  0.59            &  0.550            &  0.420   &  0.454                    \\
            \multicolumn{1}{|c}{(0.75, 0.47)}          & 0.03           &  0.958           &  0.954            &  0.864   &  0.848                    \\
            \multicolumn{1}{|c}{(1, 0.47)}             & 0.00           &  0.994               &  0.998                &  0.950    &  0.944                    \\
            \multicolumn{1}{|c}{(1.5, 0.43)}           & 0.04           &  0.978           &  0.944            &  0.828   &  0.842             \\
            \multicolumn{1}{|c}{(2.5, 0.41)}           & 0.08           &  0.440           &  0.408            &  0.322   &  0.292             \\
            \multicolumn{1}{|c}{\textbf{(3.5, 0.40)}}  & \textbf{0.10}  &  \textbf{0.102}  &  \textbf{0.092}   &  \textbf{0.08}   &  \textbf{0.078}             \\ 
            \multicolumn{1}{|c}{(4.5, 0.40)}           & 0.12           &  0.022           &  0.036   &  0.044   &  0.042            \\ 
            \multicolumn{1}{|c}{(5.5, 0.40)}           & 0.13           &  0.004           &  0.020   &  0.018   &  0.020             \\
            \multicolumn{1}{|c}{(6.5, 0.40)}           & 0.14           &  0.000              &  0.008   &  0.016   &  0.022             \\
            \hline
    \end{tabular}
    \caption{\it Scenario (B): Simulated rejection probabilities of the test \eqref{boot} in Algorithm \ref{alg2} for subgroup standard deviation $\sigma_\ell = 0.1$, significance level $\alpha = 0.1$ and equivalence threshold $\Delta = 0.1$ for different sample sizes of the subgroups and different dose group sizes.}
    \label{tab2}
\end{table}

\begin{table}[H]
    \centering
    \begin{tabular}{cc| c|c||c|c|}
    \cline{3-6}
    &  & \multicolumn{2}{c||}{ \eqref{det3a} } & \multicolumn{2}{c|}{ \eqref{det3b} } \\
     \hline
      \multicolumn{1}{|c}{$(h, ED_{50}, E_{max})$}  & $d_{\infty}$    &  $\mathcal{D}_{=}$ & $\mathcal{D}_{\neq}^{(1)}$ & $\mathcal{D}_{=}$ & $ \mathcal{D}_{\neq}^{(2)}$ \\
       \hline
            \multicolumn{1}{|c}{(2, 25, 0.47)}               & 0.03           &  0.966          &  0.946          &  0.852          &  0.782                    \\
            \multicolumn{1}{|c}{(2.25, 23, 0.47)}            & 0.04           &  0.935          &  0.850          &  0.760          &  0.696             \\
            \multicolumn{1}{|c}{(2.5, 21, 0.46)}             & 0.06           &  0.646          &  0.550          &  0.394          &  0.348             \\
            \multicolumn{1}{|c}{\textbf{(2.75, 18.5, 0.46)}} & \textbf{0.10}  &  \textbf{0.072} &  \textbf{0.052} &  \textbf{0.058} &  \textbf{0.042}             \\ 
            \multicolumn{1}{|c}{(3, 17, 0.46)}               & 0.13           &  0.004          &  0.008          &  0.008          &  0.006            \\ 
            \multicolumn{1}{|c}{(3.25, 15, 0.46)}            & 0.16           &  0.000          &  0.000          &  0.000          &  0.000             \\
            \multicolumn{1}{|c}{(3.5, 13, 0.46)}             & 0.20           &  0.000          &  0.000          &  0.000          &  0.000             \\
            \hline
    \end{tabular}
    \caption{\it Scenario (C): Simulated rejection probabilities of the test \eqref{boot} in Algorithm \ref{alg2} for subgroup standard deviation $\sigma_\ell = 0.1$, significance level $\alpha = 0.1$ and equivalence threshold $\Delta = 0.1$ for different sample sizes of the subgroups and different dose group sizes.}
    \label{tab3}
\end{table}

\subsection{Assessing similarity of all subgroups with a full population} \label{sec32}

In this section we investigate the finite sample performance of the bootstrap procedure defined in Algorithm \ref{alg3} for the same three scenarios (A), (B), (C) and dose designs defined in Section \ref{sec31}. We consider the case where all three subgroups are compared with the full population, i.e $m = k = 3$. The results of the simulation are displayed in Tables \ref{tab4}-\ref{tab6}. \medskip

The simulation results for Algorithm \ref{alg3} are in line with its theoretical validity established in Theorem \ref{boot2}, that is, the test keeps its nominal level $\alpha = 10 \%$ under the null hypothesis $d_{\infty, \infty} \geq \Delta = 0.1$ and has acceptable power under the alternative hypothesis. Overall, the finite sample performance of Algorithm \ref{alg3} in the three considered scenarios is very similar to that of Algorithm \ref{alg2} presented in the previous section. We therefore omit a more detailed discussion for the sake of brevity.



\begin{table}[H]
    \centering
    \begin{tabular}{cc| c|c || c|c|}
    \cline{3-6}
    &  & \multicolumn{2}{c||}{ \eqref{det3a} } & \multicolumn{2}{c|}{ \eqref{det3b} } \\
    \hline
       \multicolumn{1}{|c}{$(ED_{50}, E_{max})$}  & $d_{\infty,\infty}$ &  $\mathcal{D}_{=}$ & $\mathcal{D}_{\neq}^{(1)}$ & $\mathcal{D}_{=}$ & $ \mathcal{D}_{\neq}^{(2)}$ \\
       \hline
            \multicolumn{1}{|c}{(25, 0.47)}         & 0.00           &  0.998 (1.000) &  0.994 (1.000)  &  0.962 (0.986)  &  0.950 (1.000)            \\
            \multicolumn{1}{|c}{(15, 0.44)}         & 0.04           &  0.872 (0.992) &  0.832 (0.982)  &  0.690 (0.940)  &  0.664 (0.918)           \\
            \multicolumn{1}{|c}{(10, 0.42)}         & 0.07           &  0.344 (0.696) &  0.250 (0.624)  &  0.258 (0.464)  &  0.232 (0.412)            \\
            \multicolumn{1}{|c}{(8, 0.42)}          & 0.09           &  0.090 (0.252) &  0.078 (0.218)  &  0.072 (0.166)  &  0.062 (0.184)            \\ 
            \multicolumn{1}{|c}{\textbf{(7, 0.42)}} & \textbf{0.10}  &  \textbf{0.030 (0.100)} &  \textbf{0.038 (0.102)}  &  \textbf{0.022 (0.072)}  &  \textbf{0.028 (0.100)}            \\ 
            \multicolumn{1}{|c}{(6, 0.42)}          & 0.12           &  0.006 (0.016) &  0.008 (0.020)  &  0.006 (0.016)  &  0.008 (0.032)            \\
            \multicolumn{1}{|c}{(4, 0.41)}          & 0.15           &  0.000 (0.000) &  0.000 (0.000)  &  0.000 (0.004)  &  0.000 (0.004)            \\
            \multicolumn{1}{|c}{(2, 0.40)}          & 0.19           &  0.000 (0.000) &  0.000 (0.000)  &  0.000 (0.000)  &  0.000 (0.000)            \\
            \hline
    \end{tabular}
    \caption{ \it Scenario (A): Simulated rejection probabilities of the test \eqref{allboot} in Algorithm \ref{alg3} for subgroup standard deviation $\sigma_\ell = 0.1$, significance level $\alpha = 0.1$ and equivalence threshold $\Delta = 0.1$ for different sample sizes of the subgroups and different dose group sizes. Results in brackets are obtained without estimation of the Hill-parameters of the curves.}
    \label{tab4}
\end{table}

\begin{table}[H]
    \centering
    \begin{tabular}{cc| c|c||c|c|}
    \cline{3-6}
    &  & \multicolumn{2}{c||}{ \eqref{det3a} } & \multicolumn{2}{c|}{ \eqref{det3b} } \\
    \hline
       \multicolumn{1}{|c}{$(h, E_{max})$} & $d_{\infty, \infty}$ &  $\mathcal{D}_{=}$ & $\mathcal{D}_{\neq}^{(1)}$ & $\mathcal{D}_{=}$ & $ \mathcal{D}_{\neq}^{(2)}$ \\
       \hline 
            \multicolumn{1}{|c}{\textbf{(0.3,  0.47)}} & \textbf{0.10}  &  \textbf{0.038}  &  \textbf{0.044}   &  \textbf{0.036}   &  \textbf{0.060}                    \\
            \multicolumn{1}{|c}{(0.5,  0.47)}          & 0.06           &  0.454           &  0.460            &  0.384            &  0.368                    \\
            \multicolumn{1}{|c}{(0.75, 0.47)}          & 0.03           &  0.930           &  0.890            &  0.802            &  0.762                    \\
            \multicolumn{1}{|c}{(1, 0.47)}             & 0.00           &  0.992           &  0.986            &  0.926            &  0.936                 \\
            \multicolumn{1}{|c}{(1.5, 0.43)}           & 0.04           &  0.966           &  0.922            &  0.804            &  0.806             \\
            \multicolumn{1}{|c}{(2.5, 0.41)}           & 0.08           &  0.262           &  0.278            &  0.212            &  0.244             \\
            \multicolumn{1}{|c}{\textbf{(3.5, 0.40)}}  & \textbf{0.10}  &  \textbf{0.052}  &  \textbf{0.060}   &  \textbf{0.072}   &  \textbf{0.084}             \\ 
            \multicolumn{1}{|c}{(4.5, 0.40)}           & 0.12           &  0.020           &  0.008            &  0.020            &  0.032          \\ 
            \multicolumn{1}{|c}{(5.5, 0.40)}           & 0.13           &  0.002           &  0.006            &  0.008            &  0.026             \\
            \multicolumn{1}{|c}{(6.5, 0.40)}           & 0.14           &  0.002           &  0.004            &  0.012            &  0.016            \\
            \hline
    \end{tabular}
    \caption{\it Scenario (B): Simulated rejection probabilities of the test \eqref{allboot} in Algorithm \ref{alg3} for subgroup standard deviation $\sigma_\ell = 0.1$, significance level $\alpha = 0.1$ and equivalence threshold $\Delta = 0.1$ for different sample sizes of the subgroups and different dose group sizes.}
\end{table}

\begin{table}[H]
    \centering
    \begin{tabular}{cc| c|c||c|c|}
    \cline{3-6}
    &  & \multicolumn{2}{c||}{ \eqref{det3a} } & \multicolumn{2}{c|}{ \eqref{det3b} } \\
     \hline
       \multicolumn{1}{|c}{$(h, ED_{50}, E_{max})$}   & $d_{\infty, \infty}$    &  $\mathcal{D}_{=}$ & $\mathcal{D}_{\neq}^{(1)}$ & $\mathcal{D}_{=}$ & $ \mathcal{D}_{\neq}^{(2)}$ \\
       \hline
            \multicolumn{1}{|c}{(2, 25, 0.47)}               & 0.03           &  0.952          &  0.928          &  0.868          &  0.802                    \\
            \multicolumn{1}{|c}{(2.25, 23, 0.47)}            & 0.04           &  0.916          &  0.850          &  0.754          &  0.722             \\
            \multicolumn{1}{|c}{(2.5, 21, 0.46)}             & 0.06           &  0.604          &  0.538          &  0.462          &  0.354             \\
            \multicolumn{1}{|c}{\textbf{(2.75, 18.5, 0.46)}} & \textbf{0.10}  &  \textbf{0.074} &  \textbf{0.056} &  \textbf{0.062} &  \textbf{0.050}             \\ 
            \multicolumn{1}{|c}{(3, 17, 0.46)}               & 0.13           &  0.000          &  0.004          &  0.010          &  0.010            \\ 
            \multicolumn{1}{|c}{(3.25, 15, 0.46)}            & 0.16           &  0.000          &  0.000          &  0.000          &  0.000             \\
            \multicolumn{1}{|c}{(3.5, 13, 0.46)}             & 0.20           &  0.000          &  0.000          &  0.000          &  0.000             \\
            \hline
    \end{tabular}
    \caption{\it Scenario (C): Simulated rejection probabilities of the test \eqref{allboot} in Algorithm \ref{alg3} for subgroup standard deviation $\sigma_\ell = 0.1$, significance level $\alpha = 0.1$ and equivalence threshold $\Delta = 0.1$ for different sample sizes of the subgroups and different dose group sizes.}
    \label{tab6}
\end{table}

\section{Numerical example}
\label{example}

In this section we illustrate the proposed methodology analyzing a multi-regional dose finding trial example. 
We apply the methodology to the data from the dose finding study described in Section 7 of \cite{biesheuvel2002many} and re-analyzed in \cite{dette2018equivalence}. The original data set is available in the \texttt{R} package \texttt{DoseFinding} (see \cite{dosefinding}). 
In this study 369 patients with Irritable Bowel Syndrome (IBS) are investigated at five blinded doses 0 (placebo), 1, 2, 3 and 4 with the primary endpoint being a baseline-adjusted abdominal pain score. The different dose groups in the study are of approximately equal size.

\medskip
Since no information about the region is available in the original dataset, we randomly allocated the 369 patients to one of three regional subgroups (i.e., Japan, North America, and Europe) with probabilities $p_1 = 1/7$, $p_2 = 3/7$ and $p_3 = 3/7$, respectively. Accordingly, we also use these proportions in the definition of the overall population dose response function \eqref{det1}. The resulting data set consists of 58 Japanese, 141 North American and 170 European patients. We assume that the subgroup dose response functions are given by three-parametric E-Max models as defined in \eqref{emax}, where we assume a fixed Hill-coefficient $h=1$. 
The fitted dose response curves of the subgroups based on the maximum likelihood estimates are then given by
\begin{align*} 
\mu_J(d) & = 0.38 + \dfrac{0.66 \cdot d  }{d + 3.94 } \quad , \quad \mu_A(d) = 0.00 + 
\dfrac{0.68 \cdot d}{d + 1.41} \quad , \quad \mu_E(d) = -0.03 + \dfrac{ 0.90 \cdot d}{d + 0.85},
\end{align*}
representing the Japanese, North American and European subgroup, respectively. These subgroup curves, together with the corresponding dose response data, are displayed in Figure \ref{fittedcurves1} which also shows the estimated effects at the dose levels with corresponding 90 \% confidence intervals. In particular, the subgroup error variances are estimated as $\hat \sigma_J^2 = 0.58 , \ \hat \sigma_A^2 = 0.67 $ and $ \hat \sigma_E^2 = 0.72 $. Figure \ref{fittedcurves2} presents a more detailed comparison between the three curves and the average population curve. 

\medskip

\begin{figure}[H]
    \centering
    \includegraphics[width=5cm, height=4 cm]{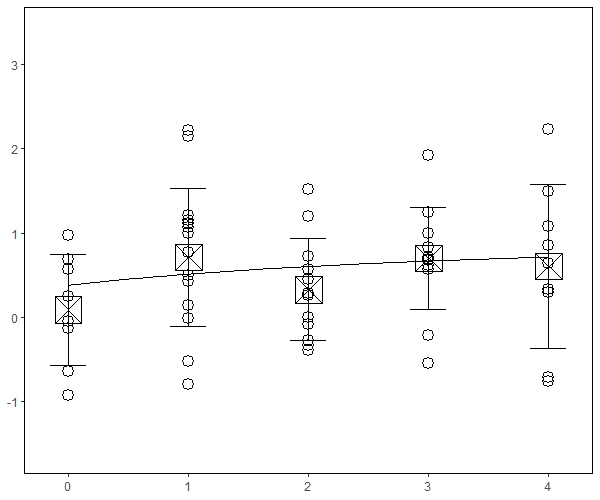}
    \includegraphics[width=5cm, height=4cm]{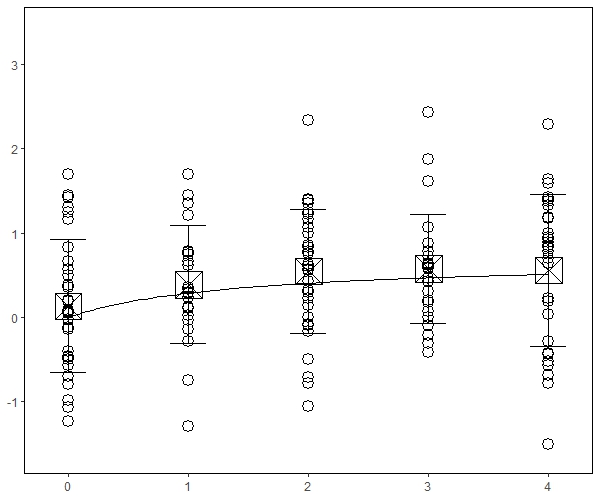}
    \includegraphics[width=5cm, height=4cm]{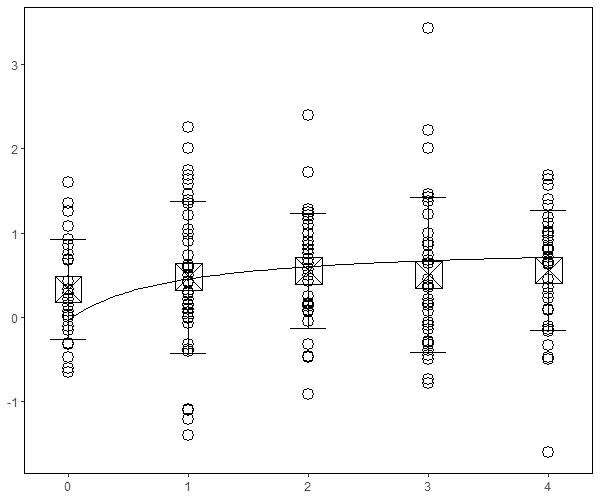}
    \caption{\it The dose response data and fitted E-Max curves for the Japanese (left), North American (middle) and European (right) subgroups in the trial. Circles represent responses, boxes represent the mean of the responses of a given dose group and bars denote $90$ \% confidence intervals for the corresponding means.}
    \label{fittedcurves1}
\end{figure}

\begin{figure}[H]
    \centering
    \includegraphics[width = 14cm, height = 8cm]{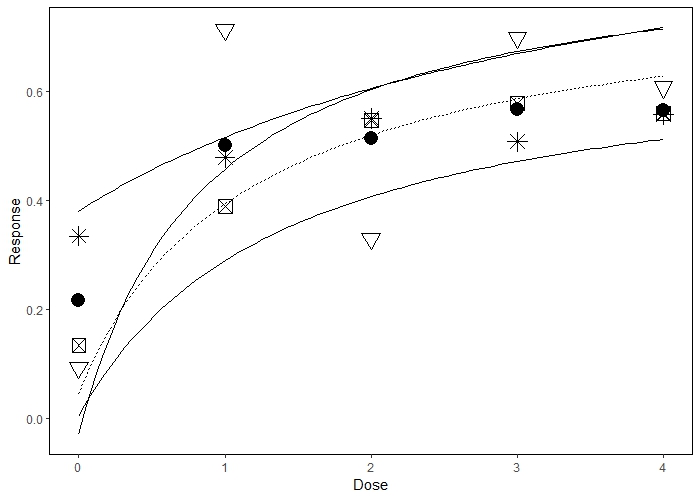}
    \caption{\it The fitted subgroup (solid) and full population (dotted) dose response curves based on the simulated dataset in the case study. The top, middle and bottom solid curve correspond to the Japanese, European and North American subgroup, respectively. Means: $\triangledown$ = Japan, $\ast$ = Europe , $\boxtimes$ = North America, $\bullet$ = Population.}
    \label{fittedcurves2}
\end{figure}

We now investigate the similarity of subgroup and population dose response curves using Algorithm \ref{alg2} and Algorithm \ref{alg3}. We use the subscripts ``J" (Japan), ``A" (North America) and ``E" (Europe) to clarify which subgroup is compared to the full population. Throughout this section we fix the similarity threshold as $\Delta = 0.4$ and the bootstrap quantiles are calculated by $B= 1000$ bootstrap replications. \medskip

First, we test similarity of the European subgroup with the full population by the test \eqref{boot} in Algorithm \ref{alg2}. The corresponding test statistic, i.e the maximum absolute deviation between the estimated European curve and estimated population curve, is $ \hat d_{\infty, E} = 0.087$. The bootstrap quantile (see step (4) in Algorithm \ref{alg2}) for $\Delta = 0.4$ and $\alpha = 0.05 \ (0.1) $ is $\hat q_{\alpha, \ts E}^* = 0.086 \ (0.113) $. We observe that $\hat d_{\infty, E} = 0.087 > 0.086 = \hat q_{0.05, E}^*$. Thus, the null hypothesis \eqref{hyp1} cannot be rejected at significance level $\alpha = 0.05$. However, since $\hat d_{\infty, E} = 0.087 < 0.113 = \hat q_{0.1, E}^*$ we can claim the dose response curves of the European subgroup and the overall population to be similar, that is $d_{\infty, E} < \Delta = 0.4$, at significance level $\alpha = 0.1$. Alternatively, we can compute the corresponding $p$-value in \eqref{pvalue1}. By assessing similarity to the European subgroup for $\Delta = 0.4$ we obtain $p_{\infty, E} = 0.067$ which is smaller than 0.1 but bigger than 0.05. \medskip

Next, we apply the test \eqref{boot} in Algorithm \ref{alg2} to assess whether the dose response curves of the Japanese, respectively, North American region and the full population are similar. The corresponding test statistics are given by $\hat d_{\infty, J} = 0.337$ and $\hat d_{\infty, A} = 0.116 $ and the bootstrap quantiles for the fixed threshold $\Delta$ and significance levels $\alpha = 0.05 \ (0.1) $ are $\hat q_{\alpha, \ts J}^* = 0.150 \ (0.208)$ and $\hat q_{\alpha, \ts A}^* = 0.042 \ (0.057) $. Therefore, at both significance levels, we do not claim the dose response relationship of the Japanese and North American subgroup to be similar to the full population, since each test statistic exceeds the corresponding quantile. Again, we come to the same conclusion by calculating the $p$-values of the corresponding tests which are given by $p_{\infty, J} = 0.156 $ and $p_{\infty, A} = 0.138 $. \medskip

Whether the subgroups are tested individually with Algorithm \ref{alg2} or simultaneously with Algorithm \ref{alg3} depends on the practitioners priorities. Individual tests with Algorithm \ref{alg2} allow for a more refined understanding of the subgroup similarities as some subgroup dose response curves might be similar to the population curve whereas others are not. However, if Algorithm \ref{alg2} is applied repeatedly, the significance level $\alpha$ might have to be reduced in each test in order to control the overall type 1 error rate which may yield conservative tests. Applying Algorithm \ref{alg3} once instead avoids the multiple testing problem on the one hand, but on the other hand only allows investigating if all considered subgroups are similar to the full population. \medskip

We now exemplarily test for similarity of all three subgroups with the full population simultaneously using the test \eqref{allboot} in Algorithm \ref{alg3}. The corresponding test statistic is calculated as $\hat d_{\infty, \infty} = 0.337 $ which is the maximum of the three individual test statistics $\hat d_{\infty, J}, \hat d_{\infty, A}$ and $\hat d_{\infty, E}$ and coincides with $\hat d_{\infty, J}$. The bootstrap quantile for $\Delta = 0.4$ (see step (4) of Algorithm \ref{alg3}) is estimated as $\hat q_{\alpha, \infty}^* = 0.189 \ts (0.235)$. The $p$-value at $\Delta = 0.4$ (see \eqref{pvalue2}) is given by $p_{\infty, \infty} = 0.201$. Hence we do not claim that all subgroup dose response functions are similar to the dose response function of the full population.

\medskip

Figure \ref{pvalues} shows the $p$-values of the two tests in relation to the equivalence threshold $\Delta$ varying in the interval $[0,1]$. As one might expect (see also Remark \ref{rem2}), larger thresholds $\Delta$ generally yield smaller $p$-values and thus a higher chance of rejecting the null hypothesis stating ``no similarity" at any given significance level. For example, if $\Delta$ is close to one, the $p$-values are nearly zero. This illustrates the general fact that, for a large enough threshold $\Delta$, any subgroup dose response curve(s) can be claimed similar to the population curve. However, clearly, the larger the threshold $\Delta$ the less meaning any such claim will carry for practical purposes. \medskip

Observing Figure \ref{pvalues} one may wonder why the $p$-values stay (approximately) constant for thresholds below the corresponding test statistic of the performed test (see, exemplarily, the dotted line in the right panel of Figure \ref{pvalues}). We note that this is a consequence of the definition of the constrained mle defined in step (2) of both Algorithms. If the test statistic exceeds the similarity threshold $\Delta$ the bootstrap data is always generated from the unconstrained mle which does not depend on $\Delta$. Consequently, the $p$-values produced for such thresholds will be (approximately) the same.

\begin{figure}[H]
    \centering
    \includegraphics[width=7cm, height=7cm]{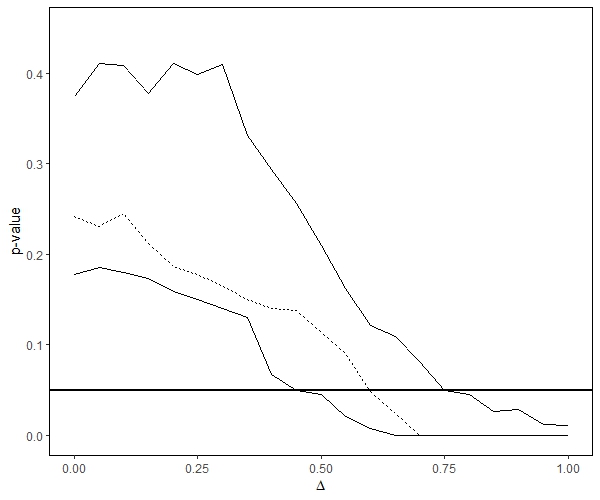}
    \includegraphics[width=7cm, height=7cm]{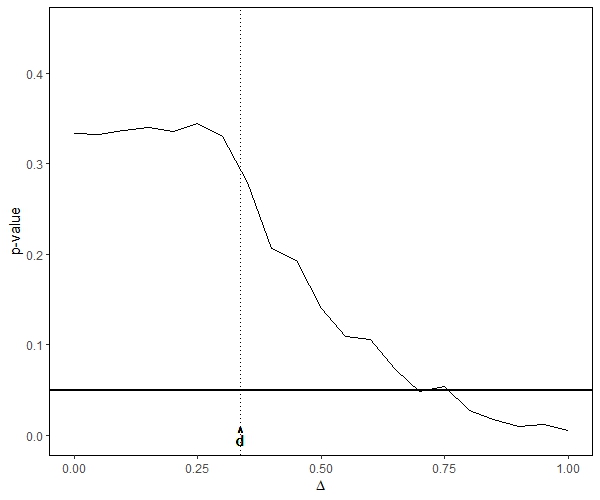}
    \caption{\it Left panel: $p$-values of the test \eqref{boot} in Algorithm \ref{alg2} assessing similarity between the Japanese (top solid), North American (dotted) and European (bottom solid) subgroup and the full population for different choices of the threshold $\Delta$. Right panel: $p$-values of the test \eqref{allboot} in Algorithm \ref{alg3} assessing similarity between all three subgroups and the full population for different choices of the threshold $\Delta$. The horizontal solid line marks a p-value of 0.05. The vertical dotted line in the right graphic marks the value of the test statistic $\hat d_{\infty, \infty}$.}
    \label{pvalues}
\end{figure}

\section{Conclusions and future research}
\label{disc}

In this paper, we develop inference tools for assessing the similarity between dose response curves of one or several subgroups and the full population in  multi-regional clinical trials. Our approach is based on a statistical test for the null hypothesis that the maximum deviation between the dose response curves  is larger than a given threshold $\Delta$. Thus rejection means that the curves of the subgroups deviate by at most $\Delta$ from the dose response curve corresponding to the full population over the full dose range. Critical values are determined by a novel parametric bootstrap under a constraint on the parameters such that the null hypothesis is satisfied.
 An essential ingredient of our approach is the specification of the threshold $\Delta$, which has to be carefully discussed for each application with the clinical team. Alternatively, our approach can be used  to define a measure of evidence for similarity with a controlled type I error rate, as pointed out in Remark \ref{rem2}.  

Our main focus in this paper is on the maximum deviation between the curves, which makes the  definition and interpretation of the threshold relatively easy. However,  other  measures such as the area between the curves might also be   useful in applications, and an interesting problem of future research is to extend our approach to such distances.  
A further important direction of future research is to extend our approach to situations, where the proportions  $p_\ell$ of the subgroups are not known (as assumed in this paper) and have to be estimated from the data. In this case 
it is not reasonable  to work with deterministic sample sizes $n_\ell $ for the subgroups and patients  have to be randomly selected from the full population for the trial. We expect that the methodology can be extended to such situations. 
Finally, in contrast to  many publications on dose finding in  multi-regional trials, we did not include a model-selection step via, for example, MCP-Mod \citep{bretz2005combining} in our methodology, but worked with   pre-selected dose response models for the subgroups from the beginning. We also look forward to generalize our  results   to account for model  uncertainty in the future.

\bigskip\noindent
\textbf{Acknowledgements:}  
This research is supported by the European Union through the European Joint Programme on Rare Diseases under the European Union's Horizon 2020 Research and Innovation Programme Grant Agreement Number 825575 and by the Deutsche Forschungsgemeinschaft (DFG, German Research Foundation) – Project-ID 499552394 – SFB 1597.

\section{Appendix}
\label{appendix}
  \def\theequation{5.\arabic{equation}}	
  \setcounter{equation}{0}

\subsection{Basic assumptions and main statements} 

In this section, we provide details about the validity of the two tests \eqref{boot} and
\eqref{allboot} for the hypotheses \eqref{hyp1} and \eqref{hyp4}. We begin stating several assumptions that are required for all theoretical results in this paper.

\medskip

\textbf{Assumption 1:} The $k$ dose response functions $\mu_1(\cdot, \beta_1), \ldots , \mu_k(\cdot, \beta_k)$ depend on unknown parameter vectors $\beta_\ell \in \R^{\gamma_\ell}, \ \ell = 1, \ldots , k$. We define $\gamma = \sum_{\ell = 1}^{k} \gamma_\ell$ and summarize these unknown vectors in a single vector $\beta = (\beta_1^\top, ... , \beta_k^\top )^\top $ which belongs to the compact parameter space $\mathcal{B} \subseteq \R^{ \gamma }$.

   \smallskip
   
     \textbf{Assumption 2:} For $\ell = 1, \ldots , k$ the subgroup dose response function $\mu_\ell(d, \beta_\ell)$ is three times continuously differentiable with respect to $\beta_\ell$ and $d$.
    \smallskip
    
     \textbf{Assumption 3:} For $\ell = 1, \ldots , k$ we have $n_\ell / n \to \kappa_\ell \in (0,1)$ and $n_{\ell,j} / n_\ell \to \kappa_{\ell,j} \in (0,1)$ for $j = 1, \ldots , r$ as $n_\ell \to \infty$. 
    
\medskip

For a statement of our first main result we introduce the random variable
\begin{equation}
    \label{det51}
   T:=   \max \Bigl\{ \max_{d \in \mathcal{E}^{+}} G(d) , \ts \max_{d \in \mathcal{E}^{-}} -G(d) \Bigr\} ,
\end{equation}
where  
 $
 \mathcal{E}^{\pm} := \{ d \in \mathcal{D} : \mu_1(d, \beta_1) - \bar \mu(d ,\beta) = \pm d_\infty \} ,
 $
the  process $G = \{ G(d) \}_{d \in \mathcal{D }} $  is defined by 
 \begin{equation}
\label{det53}     
    G(d) := \big( \tfrac{\partial (\mu_1(d, b_1) - \bar \mu(d , b)) }{\partial b} \big)^{T} \Big | _{ \ts b = \beta} \ts Z, ~~~ d \in \mathcal{D}
 \end{equation}
    and 
    $Z$ is a  centered $\gamma$-dimensional normal distributed random variable with block diagonal covariance matrix  
     \begin{align}
\label{matrix}
    \Sigma &= {\rm diag} ( \tfrac{1}{\kappa_1} \Sigma_1^{-1},  \ldots , \tfrac{1}{\kappa_k} \Sigma_k^{-1}) \in \R^{\gamma \times \gamma},
\end{align}
where
\begin{align*}
    \Sigma_\ell = \dfrac{1}{\sigma_{\ell}^2} \sum_{j=1}^{r} \kappa_{\ell, j} \Big(\dfrac{\partial}{\partial b_{\ell}} \mu_{\ell}(d_{j},b_{\ell})\Big|_{b_{\ell}=\beta_{\ell}}  \Big)\ts \Big( \dfrac{\partial}{\partial b_{\ell}} \mu_{\ell}(d_{j}, b_{\ell})\Big |_{b_{\ell}=\beta_{\ell}}\Big)^\top \in \R^{\gamma_\ell \times \gamma_\ell}, \quad \ell = 1, \ldots , k.
\end{align*}

 \begin{Thm}
 \label{boot1}
     \textit{Let Assumptions 1-3 be satisfied and assume that the random variable $T$ in \eqref{det51} has a continuous distribution function. Furthermore, let $\alpha \in (0,1)$ be small enough, such that the $\alpha$-quantile of $T$ is negative. Then the test defined by \eqref{boot} for the hypotheses \eqref{hyp1} is consistent and has asymptotic level $\alpha$. More precisely,
     \begin{itemize}
         \item[(1)] if the null hypothesis in \eqref{hyp1} is satisfied, then we have 
         \begin{align*}
            \limsup_{n \to \infty} \mathbb P( \hat d_\infty < \hat q_\alpha^* ) \leq \alpha.
         \end{align*}
         \item[(2)] if the null hypothesis in \eqref{hyp1} is satisfied and the set 
         \begin{align*}
             \mathcal{E} = \{ d \in \mathcal{D} : |\mu_1(d, \beta_1) - \bar \mu(d, \beta) | = d_\infty  \}
         \end{align*}
         consists of one point, then we have
         \begin{align*}
             \lim_{n \to \infty} \mathbb P(\hat d_\infty < \hat q_\alpha^*) = \begin{cases}
                 0,& \ d_\infty > \Delta, \\
                 \alpha,& \ d_\infty = \Delta.
             \end{cases}
         \end{align*}
         \item[(3)] if the alternative hypothesis in \eqref{hyp1} is satisfied, then we have
         \begin{align*}
          \lim_{n \to \infty}     \mathbb P( \hat d_\infty < \hat q_\alpha^* ) = 1.
         \end{align*}
     \end{itemize}
     }
 \end{Thm}

Next, we establish a similar result for the test  
\eqref{allboot}. For this purpose we introduce the random variable 
 \begin{align}
    \label{det52}
       S:=  \max \Big \{ \max_{ (i,d) \in \mathcal{\tilde E}^{+}} G( (i,d) ) , \ts \max_{ (i,d) \in \mathcal{\tilde E}^{-}} -G( (i,d) ) \Big \} ,
    \end{align}
 where
 \begin{align*}
     \mathcal{\tilde E}^{\pm} := \{ (i,d) \in \{1, ... , m \} \times \mathcal{D} : \mu_i(d, \beta_i) - \bar \mu(d ,\beta) = \pm d_{\infty, \infty} \} ,
 \end{align*}
the process $G = \{ G( (i,d) ) \}_{ (i,d) \in \{1, ... , m \} \times \mathcal{D }} $ is defined by 
    $$
    G( (i,d) ) := \big( \tfrac{\partial (\mu_i(d, b_i) - \bar \mu(d , b)) }{\partial b} \big)^{T} \Big | _{ \ts b = \beta} \ts Z, ~~~ i \in \{1, ... , m \}, \ d \in \mathcal{D}
    $$ 
    and 
    $Z$ is the centered $\gamma$-dimensional normal distributed random variable defined in \eqref{det53}.

\begin{Thm}
    \label{boot2}
    \textit{Let Assumptions 1-3 be satisfied and assume that the random variable $S$ in \eqref{det52} has a continuous distribution function. Furthermore, let $\alpha \in (0,1)$ be small enough, such that the $\alpha$-quantile of $S$ is negative. Then the test defined by \eqref{allboot} for the hypotheses \eqref{hyp4} is consistent and has asymptotic level $\alpha$. More precisely,
     \begin{itemize}
         \item[(1)] if the null hypothesis in \eqref{hyp4} is satisfied, then we have 
         \begin{align*}
            \limsup_{n \to \infty} \mathbb P( \hat d_{\infty, \infty} < \hat q_{\alpha, \infty}^* ) \leq \alpha.
         \end{align*}
         \item[(2)] if the null hypothesis in \eqref{hyp4} is satisfied and the set 
         \begin{align*}
             \tilde{\mathcal{E}} = \{ (i,d) \in \{ 1, ... , m \} \times \mathcal{D} : |\mu_i(d, \beta_i) - \bar \mu(d, \beta) | = d_{\infty, \infty}  \}
         \end{align*}
         consists of one point, then we have
         \begin{align*}
             \lim_{n \to \infty} \mathbb P(\hat d_{\infty, \infty} < \hat q_{\alpha, \infty}^*) = \begin{cases}
                 0,& \ d_{\infty, \infty} > \Delta, \\
                 \alpha,& \ d_{\infty, \infty} = \Delta.
             \end{cases}
         \end{align*}
         \item[(3)] if the alternative hypothesis in \eqref{hyp4} is satisfied, then we have
         \begin{align*}
          \lim_{n \to \infty} \mathbb P( \hat d_{\infty, \infty} < \hat q_{\alpha, \infty}^* ) = 1.
         \end{align*}
     \end{itemize}
     }
\end{Thm}

\subsection{Proof of Theorem \ref{boot1}}

\subsubsection{A preliminary result}

We begin with a preliminary result, which is the basic ingredient for the proof of Theorem  \ref{boot1}.

\begin{Thm}
\label{thm2}
\it Let Assumptions 1-3 be satisfied. Then, as $n_1, \ldots , n_k \to \infty$,  the statistic $\hat d_{\infty}$ defined in \eqref{estim1} satisfies
    \begin{align*}
        \sqrt{n}(\hat d_{\infty} - d_{\infty}) \xrightarrow{d} 
        T , 
    \end{align*}
    where the random variable $T$ is defined in \eqref{det51}.
\end{Thm}

\begin{proof}
The proof is conducted in four steps. First, observe that under Assumptions 1-3 the maximum likelihood estimate  $\hat \beta$ defined in \eqref{mle} satisfies $\sqrt{n} ( \hat \beta - \beta ) \tod Z$, where $Z \sim \mathcal{N}_{\gamma}(0, \Sigma)$ with covariance matrix $\Sigma $ defined by \eqref{matrix}.

Second,  similar arguments as given for the proof of the process convergence in equation (A.7) of \cite{dette2018equivalence} show the weak convergence of the process
\begin{align*}
    \sqrt{n} \Big \{ \big ( \mu_1(d, \hat \beta_1) - \bar \mu(d, \hat \beta) \big )  - \big (  \mu_1(d, \beta_1) - \bar \mu(d, \beta)\big  ) \Big \}_{d \in \mathcal{D}}  \tod \{ G(d) \}_{d \in \mathcal{D}}
\end{align*}
in the space $\ell^{\infty} ({\mathcal{D}})$ of bounded real-valued functions on $\mathcal{D}$ equipped with the supremum-norm $\norm{\cdot}_\infty$,
where $G$ is the Gaussian process defined in \eqref{det53}. Third, note that the mapping
\begin{align*}
    \norm{ \cdot }_\infty:   
  \begin{cases} 
    \ell^{\infty}(\mathcal{D}) &\to \R, \\ 
    g &\to \norm{g}_\infty = \sup_{d \in \mathcal{D} } | g(x) | 
      \end{cases}
\end{align*}
is directionally Hadamard differentiable with respect to $(\ell^{\infty}(\mathcal{D}), \norm{ \cdot }_\infty )$ with directional Hadamard derivative at $g_0 \in \ell^{\infty}(\mathcal{D})$ given by 
\begin{align}
\label{supderiv}
 D_{g_0}: 
  \begin{cases}
    \ell^{\infty}(\mathcal{D}) &\to \R, \\
    g  &\to 
     D_{g_0}(g) = \max \{ \max_{d \in \mathcal{E}^{+}} g(d) , \ts \max_{d \in \mathcal{E}^{-}} - g(d) \}
      \end{cases}  ,
\end{align}
where $\mathcal{E}^{\pm} := \{ d \in \mathcal{D} : g_0(d) = \pm \norm{g_0}_{\infty} \} $ (see \cite{carcamo2020directional}, Theorem 2.1). Fourth,  we apply the delta method for directionally Hadamard differentiable functions (see \cite{shapiro1991asymptotic}, Theorem 2.1) to complete  the proof.
    
\end{proof}

\subsubsection{Proof of Theorem \ref{boot1}}

    Let $\mathcal{Y}
    = \{ Y_{\ell i j} |   i=1, \ldots , n_{\ell, j} , j=1, \ldots , r, \ell =1, \ldots , k \}
    $ denote the data and define the functions  
    \begin{align*}
   g(d) &:= \mu_1(d, \beta_1) - \bar \mu(d, \beta) ,  \\
    \hat{\hat g}(d) & := \mu_1(d, \hat{\hat{\beta}}_1) - \bar \mu(d, \hat{\hat{\beta}}) , \\
     \hat g^*(d) & := \mu_1(d, \hat \beta_1^*) - \bar \mu(d, \hat \beta^*) 
      \end{align*}
     and the corresponding maximum deviations 
     $\hat d_\infty^* := \| \hat g^* \|_\infty $ and $\hat{\hat{d}}_\infty := \| \hat{\hat g} \|_\infty $. Similar arguments as given for the proof of (A.25) in \cite{dette2018equivalence} yield the process convergence
     \begin{align*}
     \sqrt{n}( \hat g^* - \hat{\hat{g}} ) \tod \{ G(d) \}_{d \in \mathcal D}
     \end{align*}
     conditionally on $\mathcal{Y}$ in probability which implies 
     \begin{align}
     \label{key2}
     \tilde T_n^*:=  
     D_g( \sqrt{n}( \hat g^* - \hat{\hat{g}} ) ) \tod D_g( G ) \overset{d}{=} T,
     \end{align}
     by the continuous mapping theorem, where $ D_g $ denotes the directional Hadamard derivative \eqref{supderiv} at the function $g$. Moreover, we have
     \begin{align*}
         \sqrt{n}( \hat{\hat{d}}_\infty - d_\infty) & = D_g(\sqrt{n}(\hat{\hat{g}} - g)) + o_{\mathbb P}(1),\\
         \sqrt{n}( \hat{d}_\infty^* - d_\infty) & = D_g(\sqrt{n}(\hat{g}^* - g)) + o_{\mathbb P}(1).
     \end{align*}
     Then, subtracting the first equation from the second and using sub-additivity of the directional Hadamard derivative yields
\begin{align}
\label{key}
    \sqrt{n}(\hat d_\infty^* - \hat{\hat{d}}_\infty) \leq \tilde T_n^* + o_{\mathbb P}(1),
\end{align}
where there is equality  if the set $ \mathcal{E}$ consists of a single point
(note that in this case  the directional Hadamard derivative is linear). 
Building on these results we can derive part (1) and (2) of Theorem \ref{boot1}. In the case $d_\infty > \Delta$, note that
\begin{align*}
    \mathbb P( \hat d_\infty < \hat q_\alpha^*) =  \mathbb P( \hat d_\infty <\hat q_\alpha^*, \ \hat d_\infty \geq \Delta ) + \mathbb P( \hat d_\infty <\hat q_\alpha^*, \ \hat d_\infty < \Delta ),
\end{align*}
where the first probability converges to zero, 
which follows by similar arguments as given on page
$727$ in \cite{dette2018equivalence}. The second term converges to zero, since, by  Theorem \ref{thm2}, $\hat d_\infty \topr d_\infty > \Delta$ . In the case $d_\infty = \Delta$, note that
\begin{align*}
    \mathbb P( \hat d_\infty <\hat q_\alpha^*) =  \mathbb P( \hat d_\infty < \hat q_\alpha^*, \ \hat{\hat{d}}_\infty = \Delta ) + \mathbb P( \hat d_\infty <\hat q_\alpha^*, \ \hat{\hat{d}}_\infty > \Delta ),
\end{align*}
where the first probability sequence is asymptotically bounded above by (or equal to) $\alpha$, because of \eqref{key2}  and  \eqref{key} and the second probability sequence converges to zero, since $q_\alpha < 0$. Finally, in the case $d_\infty < \Delta$
statement (3) follows by the same arguments as given for the proof of (3.18) in Theorem 2 in \cite{dette2018equivalence}.



\subsection{Proof of Theorem \ref{boot2}}
\label{alg2proof}

The proof is analogous to the proof of Theorem \ref{boot1} and only needs some additional adjustments. We start by deriving the asymptotic error distribution of the estimator $\hat d_{\infty, \infty}$.

\subsubsection{A preliminary result}

\begin{Thm}
    \label{thm3}
    \it Let Assumptions 1-3 be satisfied. Then, as $n_1, \ldots , n_k \to \infty$,  the statistic $\hat d_{\infty, \infty}$ defined in \eqref{maxall} satisfies
    \begin{align*}
        \sqrt{n}(\hat d_{\infty, \infty} - d_{\infty, \infty}) \xrightarrow{d}  S, 
    \end{align*}
where the random variable $S$ is defined in \eqref{det52}.
\end{Thm}

\begin{proof}
    We can copy the proof of Theorem \ref{thm2} and only need to change the second and third step, the first and fourth step can be left unchanged. To this end, only note that by the functional delta method
   \begin{align*}
    \sqrt{n} \Big \{ \big ( \mu_i(d, \hat \beta_i) - \bar \mu(d, \hat \beta) \big )  - \big (  \mu_i(d, \beta_i) - \bar \mu(d, \beta)\big  ) \Big \}_{ (i,d) \in \{1, ... , m\} \times \mathcal{D}}  \tod \{ G( (i,d) ) \}_{(i,d) \in \{1, ... , m\} \times \mathcal{D}}
\end{align*}
holds true in the space $\ell^{\infty} (\{1, ... , m\} \times {\mathcal{D}})$ of bounded real-valued functions on $\{1, ... , m\} \times {\mathcal{D}}$ equipped with the supremum-norm
$$
\norm{ g }_\infty = \sup_{(i,d) \in \{1, ... , m\} \times \mathcal{D}} | g(i,d) |
$$ 
and that the supremum-norm is  directionally Hadamard differentiable on the space $\ell^{\infty} (\{1, ... , m\} \times {\mathcal{D}})$ with derivative at $g_0$ given by
\begin{align*}
 D_{g_0}: 
  \begin{cases}
    \ell^{\infty}(\{1, ... , m\} \times {\mathcal{D}}) &\to \R, \\
    g  &\to 
     D_{g_0}(g) = \max \{ \max_{ (i,d) \in \mathcal{E}^{+}} g(i,d) , \ts \max_{ (i,d) \in \mathcal{E}^{-}} - g(i,d) \}
      \end{cases}  ,
\end{align*}
where $\mathcal{E}^{\pm} := \{ (i,d) \in \{1, ... , m\} \times {\mathcal{D}} : g_0(i,d) = \pm \norm{g_0}_{\infty} \} $.

\end{proof}

\subsubsection{Proof of Theorem \ref{boot2} }

Defining the processes $ g((i,d)) := \mu_i(d, \beta_i) - \bar \mu(d, \beta) $, $ \hat{\hat g}((i,d)) := \mu_i(d, \hat{\hat{\beta}}_i) - \bar \mu(d, \hat{\hat{\beta}}) $, $ \hat g^*((i,d)) := \mu_i(d, \hat \beta_i^*) - \bar \mu(d,  \hat \beta^*) $ as well as $\hat d_{\infty, \infty}^* := \| \hat g^* \|_\infty $ and $\hat{\hat{d}}_{\infty, \infty} := \| \hat{\hat g} \|_\infty $ over the set $\{1, ... , m \} \times \mathcal{D}$, we can use the same arguments as given in the proof of Theorem \ref{boot1}, where we employ Theorem \ref{thm3} instead of Theorem \ref{thm2}.

\bibliographystyle{apalike}
\bibliography{bibliography.bib}
      
\end{document}